\documentclass[a4paper,12pt]{article}
\usepackage[cp1251]{inputenc}
\usepackage[russian, ukrainian,english]{babel}
\usepackage{amsmath, amsthm, amsfonts, amssymb}
\usepackage{graphicx}

\theoremstyle{plain}
\newtheorem{theorem}{Theorem}[section]

\newtheorem{lemma}{Lemma}[section]
\newtheorem{corollary}{Corollary}[section]

\begin{document}

\begin{center}

\begin{center}
		{\bf On a Brownian motion conditioned to stay in an open set}
		
		\vskip20pt

		Georgii V. Riabov
		
		\vskip20pt	
		
		Institute of Mathematics, NAS of Ukraine

			\end{center}
		
		\vskip20pt

 \begin{abstract}
 Distribution of a Brownian motion conditioned to start from the boundary of an open set $G$ and to stay in $G$ for a finite period of time is studied. Characterizations of such distributions in terms of certain singular stochastic differential equations are obtained. Results are applied to the study of boundaries of clusters in some coalescing stochastic flows on $\mathbb{R}.$
\end{abstract}

\begin{quote} {\footnotesize

}

\end{quote}

\end{center}

\section{Introduction}

Let $B=\{B(t)\}_{t\in [0,T]}$ be a standard $\mathbb{R}^d-$valued Brownian motion. Given an open set $G\subset \mathbb{R}^d$ denote by $\tau_G=\inf\{t>0: B(t)\not\in G\}$ the first exit time of $B$ from the set $G.$ In this paper we study the distribution of $B$ conditioned on the event $\{\tau_G>T\},$ where $T>0$ is a fixed positive time. Denote this distribution by $\nu_{x,T}(\cdot;G),$ where $B(0)=x$ is the starting point. Let $\mathcal{C}^d_T$ be the space of continuous functions $w:[0,T]\to \mathbb{R}^d$ endowed with the sup-norm and a Borelian $\sigma-$field $\mathcal{B}(\mathcal{C}^d_T).$ Then 
$$
\nu_{x,T}(\Delta;G)=\mathbb{P}(B\in \Delta|B(0)=x, \tau_G>T), \ \Delta\in\mathcal{B}(\mathcal{C}^d_T).
$$
The measure $\nu_{x,T}$ is not well-defined when $x\not\in G,$ as the event $\{B(0)=x,\tau_G>T\}$ can be of probability zero. However, if the set $G$ is sufficiently regular and $x$ is a boundary point of $G,$ the measure $\nu_{x,T}$ is well-defined as a weak limit \cite[Th. 4.1]{Garbit}
$$
\nu_{x,T}(\cdot;G)=\lim_{y\to x, y\in G}\nu_{y,T}(\cdot;G).
$$
In the paper we characterize the measure $\nu_{x,T}(\cdot;G)$ in terms of a singular SDE. Precisely, introduce the function
\begin{equation}
\label{28_08_eq3}
\gamma_G(t,y)=\mathbb{P}(\tau_G>t|B(0)=y), \ t>0, y\in G,
\end{equation}
and consider the following problem

\begin{equation}
\label{eq1}
\begin{cases}
dY(t)=\nabla_y\log\gamma_G(T-t,Y(t))dt+dW(t), \\
Y(0)=x, \\
Y(t)\in G \mbox{ for a.a. } t\in (0,T),
\end{cases}
\end{equation}
where $W$ is a standard Brownian motion in $\mathbb{R}^d.$  The main result of the paper is the following.

\begin{theorem}
\label{thm1} Let $G\subset \mathbb{R}^d$ be an open convex set, $x\in \partial G,$ and the boundary of $G$ is $C^2$ in the neighborhood of $x.$ Then the problem \eqref{eq1} has a unique strong solution. The distribution of this solution coincides with $\nu_{x,T}(\cdot,G).$

\end{theorem}

The result was motivated by the study of coalescing stochastic flows on the real line. By a coalescing stochastic flow on the real line we understand a family $\{\psi_{s,t}:-\infty<s\leq t<\infty\}$ of measurable random mappings of $\mathbb{R},$ such that:

\begin{enumerate}
	
\item For all $r\leq s\leq t,$ $x\in\mathbb{R},$  $\omega\in \Omega$ 
$$
\psi_{s,t}(\omega,\psi_{r,s}(\omega,x))=\psi_{r,t}(\omega,x)
$$
and $\psi_{s,s}(\omega,x)=x.$

\item For all $t_1\leq \ldots\leq t_n,$ $x_1,\ldots,x_m\in \mathbb{R}$ random vectors 
$$
(\psi_{t_1,t_2}(x_1),\ldots,\psi_{t_1,t_2}(x_m)),\ldots, (\psi_{t_{n-1},t_n}(x_1),\ldots,\psi_{t_{n-1},t_n}(x_m))
$$
are independent.

\item For all $s\leq t,$ $h\in \mathbb{R},$ $x_1,\ldots,x_m\in\mathbb{R}$ random vectors 
$$
(\psi_{s,t}(x_1),\ldots,\psi_{s,t}(x_m)) \mbox{ and } (\psi_{s+h,t+h}(x_1),\ldots,\psi_{s+h,t+h}(x_m))
$$
are equally distributed.

\item For all $s,x\in\mathbb{R},$ $\omega\in\Omega,$ functions
$$
t\to \psi_{s,t}(x,\omega), \ t\geq s
$$ 
are continuous.

\item There exist $x\ne y$ such that 
$$
\mathbb{P}(\exists t>0: \psi_{0,t}(x)=\psi_{0,t}(y))>0.
$$
	
\end{enumerate}

With a stochastic flow $\psi$ we associate the family of $\sigma-$fields
$$
\mathcal{F}^\psi_{s,t}=\sigma(\{\psi_{u,v}(x):s\leq u\leq v\leq t, x\in\mathbb{R}\}), s\leq t.
$$
For general properties of stochastis flows we refer to \cite{LJR}. In our previous works \cite{Riabov_RDS,Riabov_duality,DRS} properties of clusters in certain coalescing stochastic flows were investigated. To illustrate the results and related questions, let us consider the Arratia flow on $\mathbb{R}.$ 
A stochastic flow $\{\psi_{s,t}:-\infty<s\leq t<\infty\}$ is called the Arratia flow, if for all $s\in\mathbb{R},$ $n\geq 1$ and $x=(x_1,\ldots,x_n)\in\mathbb{R}^n$ processes 
$$
W_j(t)=\psi_{s,s+t}(x_j), t\geq 0, 1\leq j\leq n
$$
are $(\mathcal{F}^\psi_{s,s+t})_{t\geq 0}-$Brownian motions  with joint quadratic variation given by 
$$
\langle W_i,W_j\rangle (t)=(t-\tau_{ij})_+, \tau_{ij}=\inf\{t\geq 0: W_i(t)=W_j(t)\}.
$$
Informally, the Arratia flow describes the joint motion of a continuum family of stochastic processes that start at every moment of time from every point of the real line, each process is a standard Brownian motion, every two trajectories move independently before they meet each other, at the meeting time trajectories coalesce into one Brownian motion. For the existence of the Arratia flow and its properties we refer to \cite{ LJR, Riabov_RDS, Arratia1, Arratia2, TW}. For fixed $s<t$ consider the random mapping $\psi_{s,t}:\mathbb{R}\times \Omega\to \mathbb{R}$ from the Arratia flow. With probability $1$ it is an increasing piecewise constant function \cite{Arratia1}. The distribution of its range $\psi_{s,t}(\mathbb{R})$ as a point process on the real line was described in \cite{TZ}. Consider a point $\zeta\in \psi_{0,T}(\mathbb{R}).$ At every time $t\in [0,T]$ there exists a non-empty interval of points that have coalesced into $\zeta$ at time $T:$
$$
K_\zeta(t)=\{x\in \mathbb{R}:\psi_{T-t,T}(x)=\zeta\}, \ 0\leq t \leq T.
$$
We refer to the set $K_\zeta=\cup_{t\in[0,T]}(\{T-t\}\times K_\zeta(t))$ as to the cluster with the vertex $\zeta.$ For fixed $t\in [0,T]$ the family $\{K_\zeta(t):\zeta\in \psi_{0,T}(\mathbb{R})\}$ is a partition of $\mathbb{R}.$ Given a segment $[a,b]$ let $N_T(a,b)$ denote the number of clusters that were formed by trajectories started at time $0$ from $[a,b],$ i.e. $N_T(a,b)$ is the cardinality of the set $\{\zeta\in \psi_{0,T}(\mathbb{R}): K_\zeta(T)\cap [a,b]\ne\emptyset\}.$  The distribution of $N_T(a,b)$ was found in \cite{Fomichov}. We are interested in the distribution of boundary processes
$$
\alpha_\zeta(t)=\inf K_\zeta (t), \ \beta_\zeta(t)=\sup K_\zeta (t).
$$
In different terms, $(\alpha_\zeta(t),\beta_\zeta (t))$ is the largest open interval, where $\psi_{T-t,T}(x)=\zeta.$ Hence the distribution of boundary processes is needed in order to describe the distribution of a random mapping $\psi_{s,t}$ completely. We apply Theorem \ref{thm1} to characterize the distribution of the pair $(\alpha_\zeta,\beta_\zeta).$ Namely, in section 4 we prove 

\begin{theorem} 
\label{thm2}
Let $H=\{y\in \mathbb{R}^2:y_1<y_2\}$ and $x\in \mathbb{R}.$ Conditionally on $\{\zeta=x\}$ the distribution of the pair $\{(\alpha_\zeta(t),\beta_\zeta (t))\}_{t\in [0,T]}$ coincides with the distribution of the solution $\{Y(t)\}_{t\in [0,T]}$ of the problem 
$$
\begin{cases}
dY_1(t)=-\frac{e^{-\frac{(Y_2(t)-Y_1(t))^2}{2(T-t)}}}{\sqrt{4(T-t)}E(\frac{Y_2(t)-Y_1(t)}{\sqrt{2(T-t)}})}dt+dW_1(t), \\
dY_2(t)=\frac{e^{-\frac{(Y_2(t)-Y_1(t))^2}{2(T-t)}}}{\sqrt{4(T-t)}E(\frac{Y_2(t)-Y_1(t)}{\sqrt{2(T-t)}})}dt+dW_2(t), \\
Y_1(0)=Y_2(0)=x, \\
Y_1(t)<Y_2(t) \mbox{ for a.a. } t\in (0,T),
\end{cases}
$$
where $W$ is a standard $\mathbb{R}^2-$valued Brownian motion and $E(x)=\int^x_0 e^{-\frac{u^2}{2}}du.$
\end{theorem}

The conditional distribution of boundary processes needs to be defined rigorously, as the event $\{\zeta=x\}$ is of probability zero. This is done in section 4 using duality theory for the Arratia flow. Also in section 4 we consider Arratia flows with drift. Let $a:\mathbb{R}\to \mathbb{R}$ be a Lipschitz function. The Arratia flow with drift $a$ is a stochastic flow $\psi$ such that each trajectory $t\to \psi_{s,t}(x)$ is a weak solution of the stochastic differential equation 
$$
d\psi_{s,t}(x)=a(\psi_{s,t}(x))dt+dw_{s,x}(t),
$$
every two trajectories move independently before they meet each other, at the meeting time trajectories coalesce (see section 4.2 for the precise definition).  In \cite{DRS} it was proved that if $a'(x)\leq -\lambda<0$ a.s., then there exists a unique stationary process $\{\eta_t\}_{t\in \mathbb{R}}$ such that for all $s\leq t,$ $\psi_{s,t}(\eta_s)=\eta_t.$ At every moment $t\geq 0$ there exists an interval of points that have coalesced into $\eta_0$ at time $0:$
$$
K_0(t)=\{x\in \mathbb{R}:\psi_{-t,0}(x)=\eta_0\}, \ t\geq 0.
$$
The set $K_0=\cup_{t\geq 0}(\{-t\}\times K_0(t))$ will be called the infinite cluster with the vertex $\eta_0.$ The theorem \ref{thm18_2} (section 4.2) describes the conditional distribution of processes $\alpha_0(t)=\inf K_0 (t),$ $\beta_0(t)=\sup K_0 (t)$
conditioned on the event $\{\eta_0=x\}.$

The paper is organized as follows. Our approach is based on a carefull analysis of a Brownian meander - a particular case of Theorem \ref{thm1}, that corresponds to $d=1,$ $G=(0,\infty),$ $x=0.$ As a corollary, we recover the result of \cite{Imhof} on the mutual equivalence between the distribution of the Brownian meander and the distribution of the three-dimensional Bessel process. In section 3 we prove Theorem \ref{thm1} in full generality, by adapting the approach of \cite{Garbit}. Finally, in section 4 we apply the result to the distribution of boundaries of clusters  in the Arratia flow, and obtain analogous results for an unbounded cluster in the Arratia flow with drift \cite{DRS}.

\section{Brownian meander}

Let $P_x$ be the Wiener measure on $\mathcal{C}^1_T,$ i.e. the distribution of an $\mathbb{R}-$valued Brownian motion $B=\{B(t)\}_{t\in [0,T]}$ conditioned to start from $x\in \mathbb{R}.$  Expectation with respect to the measure $P_x$ will be denoted by $E_x.$ Denote $\mathbb{R}_+=(0,\infty).$ By the distribution of the Brownian meander we understand the measure $\nu_{0,T}(\cdot,\mathbb{R}_+).$  Informally, it is the restriction of the Wiener measure $P_0$ to the set of trajectories 
$$
A=\{w\in \mathcal{C}^1_T:w(t)>0, 0<t\leq T\}.
$$
As it was mentioned in the Introduction, $\nu_{0,T}(\cdot,\mathbb{R}_+)$ is rigorously defined as a weak limit \cite[Th. (2.1)]{DIM}
$$
\nu_{0,T}(\cdot,\mathbb{R}_+)=\lim_{y\to 0+}\nu_{y,T}(\cdot,\mathbb{R}_+),
$$
where now $\nu_{y,T}(\Delta,\mathbb{R}_+)=\frac{P_y(\Delta \cap A)}{P_y(A)}$. Introduce the function
$$
\gamma_{\mathbb{R}_+}(t,y)=P_y(\min_{s\in [0,t]} w(s)>0), \ t>0, y>0.
$$
Precisely,
\begin{equation}
\label{eq12_8_2}
\gamma_{\mathbb{R}_+}(t,y)=\sqrt{\frac{2}{\pi}}\int^{\frac{y}{\sqrt{t}}}_0 e^{-\frac{z^2}{2}}dz.
\end{equation}
Consider the following problem

\begin{equation}
\label{eq2}
\begin{cases}
dY(t)=\partial_y\log\gamma_{\mathbb{R}_+}(T-t,Y(t))dt+dW(t) \\
Y(0)=0 \\
Y(t)>0 \mbox{ for a.a. } t\in (0,T)
\end{cases}
\end{equation}
where $W$ is a standard $\mathbb{R}-$valued Brownian motion.

\begin{theorem}
\label{thm3} The problem \eqref{eq2} has a unique strong solution. The distribution of this solution coincides with the distribution of the Brownian meander $\nu_{0,T}(\cdot,\mathbb{R}_+).$

\end{theorem}

\begin{proof} For a fixed $y>0$ the measure $\nu_{y,T}(\cdot,\mathbb{R}_+)$ is absolutely continuous with respect to the Wiener measure $P_y.$ The corresponding Radon-Nikodym density is
$$
\frac{d \nu_{y,T}(\cdot,\mathbb{R}_+)}{dP_y}=\frac{1_{\min_{t\in [0,T]} w(t)>0 }}{\gamma_{\mathbb{R}_+}(T,y)}.
$$
We will apply the Girsanov theorem to the measure  $\nu_{y,T}(\cdot,\mathbb{R}_+).$ Let $(\mathcal{F}_t)_{t\in [0,T]}$ be the canonical filtration on the space $\mathcal{C}^1_T.$ We introduce the martingale associated with the Radon-Nikodym density $\frac{d \nu_{y,T}(\cdot,\mathbb{R}_+)}{dP_y}:$
$$
\rho_t=E_y\left(\frac{d \nu_{y,T}(\cdot,\mathbb{R}_+)}{dP_y}\bigg|\mathcal{F}_t\right).
$$
By the Markov property, 
$$
\rho_t=\frac{P_y(\min_{s\in [0,T]} w(s)>0 |\mathcal{F}_t)}{\gamma_{\mathbb{R}_+}(T,y)}=\frac{1_{\min_{s\in[0,t]} w(s)>0}\gamma_{\mathbb{R}_+}(T-t,w(t))}{\gamma_{\mathbb{R}_+}(T,y)} \ P_y\mbox{-a.s.}
$$

The Clark representation for the density equals  \cite[Lemma 1]{Shiryaev_Yor}
\begin{equation}
\label{27_08_eq1}
\rho_T=1+\int^T_0 1_{\min_{s\in [0,t]} w(s)>0} \frac{\partial_y \gamma_{\mathbb{R}_+}(T-t,w(t))}{\gamma_{\mathbb{R}_+}(T,y)}dw(t)  \ P_y\mbox{-a.s.}
\end{equation}
Since similar results  will be used several times in the paper, we give a proof of \eqref{27_08_eq1}. 

Recall that the function $\gamma_{\mathbb{R}_+}(t,y)$ satisfies the heat equation 
$$
\partial_t \gamma_{\mathbb{R}_+}(t,y)=\frac{1}{2}\partial^2_y \gamma_{\mathbb{R}_+}(t,y), \ t,y>0.
$$
Let $\sigma=\inf\{t\geq 0: w(t)=0\}.$ Applying the It\^o formula to the process 
$$
t\to \gamma_{\mathbb{R}_+}(T-t\wedge \sigma,w(t\wedge \sigma)), t\geq 0,
$$
we get
$$
\gamma_{\mathbb{R}_+}(T-T\wedge \sigma,w(T\wedge \sigma))=\gamma_{\mathbb{R}_+}(T,y)+\int^{T\wedge \sigma}_0 \partial_y \gamma_{\mathbb{R}_+}(T-t,w(t))dw(t).
$$
Observe that 
$$
\gamma_{\mathbb{R}_+}(T-T\wedge \sigma,w(T\wedge \sigma))=\begin{cases}
\gamma_{\mathbb{R}_+}(T- \sigma,w(\sigma))=0, \  \sigma<T \\
\gamma_{\mathbb{R}_+}(0,w(T))=1, \ \sigma>T
\end{cases}.
$$
Consequently,
$$
1_{\min_{s\in [0,T]} w(s)>0}=1_{\sigma>T}=\gamma_{\mathbb{R}_+}(T,y)+\int^{T\wedge \sigma}_0 \partial_y \gamma_{\mathbb{R}_+}(T-t,w(t))dw(t).
$$
Dividing by  $\gamma_{\mathbb{R}_+}(T,y)$ we recover \eqref{27_08_eq1}.

Let us denote $h_t= 1_{\min_{s\in [0,t]} w(s)>0} \frac{\partial_y \gamma_{\mathbb{R}_+}(T-t,w(t))}{\gamma_{\mathbb{R}_+}(T,y)},$ so that $\rho_T=1+\int^T_0 h_tdw(t).$
By the Girsanov theorem \cite[Th. (1.12), Ch. VIII]{RY} under the measure $\nu_{y,T}(\cdot,\mathbb{R}_+)$ the process 
$$
B_y(t)=w(t)-\int^{t}_0 \frac{h_s}{\rho_s}ds, \ 0\leq t\leq T.
$$
is a Brownian motion. Observe that $1_{\min_{s\in [0,T]} w(s)>0}=1$ a.s. with respect to the measure $\nu_{y,T}(\cdot,\mathbb{R}_+).$ Hence,
$$
\frac{h_s}{\rho_s}= \frac{\partial_y \gamma_{\mathbb{R}_+}(T-s,w(s))}{\gamma_{\mathbb{R}_+}(T-s,w(s))}=\partial_y \log \gamma_{\mathbb{R}_+}(T-s,w(s)) \ \ \ \nu_{y,T}(\cdot,\mathbb{R}_+)-\mbox{a.s.,}
$$
and under the measure  $\nu_{y,T}(\cdot,\mathbb{R}_+)$ the process 
$$
B_y(t)=w(t)-\int^{t}_0 \partial_y \log \gamma_{\mathbb{R}_+}(T-s,w(s))ds, \ 0\leq t\leq T,
$$
is a Brownian motion. Redenoting $w$ with $Y_y$ we can reformulate the conclusion as follows: for every $y>0$ on some probability space there is a pair of processes $(Y_y,B_y),$ such that 

\begin{itemize}
\item $\{B_y(t)\}_{t\in [0,T]}$ is a Brownian motion with the starting point $B_y(0)=y$;

\item the distribution of $\{Y_y(t)\}_{t\in [0,T]}$ is $\nu_{y,T}(\cdot,\mathbb{R}_+);$

\item for all $t\in [0,T]$ $Y_y(t)>0;$

\item for all $t\in [0,T]$
\begin{equation}
\label{eq10_8_1}
Y_y(t)=\int^t_0  \partial_y \log\gamma_{\mathbb{R}_+}(T-s,Y_y(s))ds + B_y(t).
\end{equation}

\end{itemize}

By \cite[Th. (2.1)]{DIM} $Y_y\xrightarrow{d}\nu_{0,T}(\cdot,\mathbb{R}_+).$ Hence, the family of processes $\{(Y_y,B_y):y\in (0,1]\}$ is weakly relatively compact. Applying the Skorokhod theorem \cite[Th. 4.30]{K} we can construct a sequence $y_n\to 0$ and copies of processes $\{(Y_{y_n},B_{y_n}):n\geq 1\}$ defined on the same probability space, such that
$$
(Y_{y_n},B_{y_n})\to (Y_{0},B_{0}) \ \mbox{ a.s. in  } \mathcal{C}_T(\mathbb{R}^2).
$$
We will check that 
$$
\partial_y \log\gamma_{\mathbb{R}_+}(T-s,Y_{y_n}(s))\to \partial_y \log\gamma_{\mathbb{R}_+}(T-s,Y_0(s))ds \ \mbox{ in } L^1(\Omega\times[0,T]).
$$
To prove this convergence we will use Scheff\' e's lemma \cite{Scheffe}. The lemma can be applied since $\partial_y\log\gamma_{\mathbb{R}_+}(t,y)>0$ for $t,y>0.$ Thus,  it is enough to show 
\begin{equation}
\label{eq10_8_2}
\lim_{n\to \infty}\mathbb{E}\int^T_0 \partial_y\log\gamma_{\mathbb{R}_+}(T-s,Y_{y_n}(s))ds =  \mathbb{E}\int^T_0 \partial_y\log\gamma_{\mathbb{R}_+}(T-s,Y_{0}(s))ds<\infty.
\end{equation}
Next two results allow to control the behaviour of integrals  in \eqref{eq10_8_2} near boundaries.

\begin{lemma}
\label{lem1} For each $t\in (0,T)$
$$
\lim_{n\to\infty}\mathbb{E}\int^t_0 \partial_y\log\gamma_{\mathbb{R}_+}(T-s,Y_{y_n}(s))ds=\sqrt{T}\int^\infty_0 y^2\gamma_{\mathbb{R}_+}(T-t,\sqrt{t}y)e^{-y^2/2}dy.
$$
Expression on the right-hand side is a continuous function of $t\in[0,T].$
\end{lemma}

\begin{proof}
We make use of the relation \eqref{eq10_8_1}: 
$$
\mathbb{E}\int^t_0 \partial_y\log\gamma_{\mathbb{R}_+}(T-s,Y_{y_n}(s))ds=\mathbb{E}Y_{y_n}(t)-\mathbb{E}B_{y_n}(t)=\mathbb{E}Y_{y_n}(t)-y_n.
$$
Further,
$$
\mathbb{E}Y_{y_n}(t)=\mathbb{E}\left (B_{y_n}(t)\bigg|\min_{s\in[0,T]}B_{y_n}(s)>0\right)=\frac{\mathbb{E}B_{y_n}(t)1_{\min_{s\in[0,T]}B_{y_n}(s)>0}}{\gamma_{\mathbb{R}_+}(T,y_n)}=
$$
$$
=\frac{\mathbb{E}B_{y_n}(t)1_{\min_{s\in[0,t]}B_{y_n}(s)>0}\gamma_{\mathbb{R}_+}(T-t,B_{y_n}(t))}{\gamma_{\mathbb{R}_+}(T,y_n)}=
$$
$$
=\frac{\int^\infty_0 y\gamma_{\mathbb{R}_+}(T-t,y)\frac{1}{\sqrt{2\pi t}}e^{-\frac{(y-y_n)^2}{2t}}(1-e^{-\frac{2yy_n}{t}})dy}{\int^\infty_0 \frac{1}{\sqrt{2\pi T }}e^{-\frac{(y-y_n)^2}{2T}}(1-e^{-\frac{2yy_n}{T}})dy}=
$$
$$
=\frac{\int^\infty_0 y\gamma_{\mathbb{R}_+}(T-t,y)\frac{1}{\sqrt{2\pi t}}e^{-\frac{(y-y_n)^2}{2t}}\frac{1-e^{-\frac{2yy_n}{t}}}{2y_n}dy}{\int^\infty_0 \frac{1}{\sqrt{2\pi T }}e^{-\frac{(y-y_n)^2}{2T}}\frac{1-e^{-\frac{2yy_n}{T}}}{2y_n}dy}.
$$
Hence, by the Dominated Convergence Theorem,
$$
\lim_{n\to \infty} \mathbb{E}Y_{y_n}(t)=\frac{\int^\infty_0 y^2 \gamma_{\mathbb{R}_+}(T-t,y)t^{-3/2}e^{-\frac{y^2}{2t}}dy}{\int^\infty_0 y T^{-3/2}e^{-\frac{y^2}{2T}}dy}=\sqrt{T}\int^\infty_0 y^2\gamma_{\mathbb{R}_+}(T-t,\sqrt{t}y) e^{-\frac{y^2}{2}}dy.
$$

\end{proof}

Applying Dini's theorem we deduce the corollary from the lemma \ref{lem1}.

\begin{corollary} Functions $f_n(t)= \mathbb{E}\int^t_0  \partial_y\log\gamma_{\mathbb{R}_+}(T-s,Y_{y_n}(s))ds,$ $0\leq t\leq T,$ are equicontinuous on $[0,T].$ In particular,
\label{cor1}
$$
\lim_{\delta\to 0}\sup_{n\geq 1}\left(\mathbb{E}\int^\delta_0  \partial_y\log\gamma_{\mathbb{R}_+}(T-s,Y_{y_n}(s))ds + \mathbb{E}\int^T_{T-\delta} \partial_y\log\gamma_{\mathbb{R}_+}(T-s,Y_{y_n}(s))ds\right)=0.
$$
\end{corollary}

Now we return to the proof of the theorem \ref{thm3}.  By corollary \ref{cor1} it is enough to check the convergence 
$$
\lim_{n\to \infty}\mathbb{E}\int^{T-\delta}_\delta \partial_y\log\gamma_{\mathbb{R}_+}(T-s,Y_{y_n}(s))ds =\mathbb{E}\int^{T-\delta}_\delta \partial_y\log\gamma_{\mathbb{R}_+}(T-s,Y_{0}(s))ds,
$$
for any  $\delta\in (0,T).$  This in turn will follow from the uniform integrability condition \cite[Ch. 4]{K}
\begin{equation}
\label{eq10_8_3}
\sup_{n\geq 1}\mathbb{E}\int^{T-\delta}_\delta \left( \partial_y\log\gamma_{\mathbb{R}_+}(T-s,Y_{y_n}(s))\right)^{3/2}ds<\infty.
\end{equation}
In order to verify \eqref{eq10_8_3} we make use of the estimate 
$$
\partial_y\log\gamma_{\mathbb{R}_+}(t,y)=\frac{e^{-\frac{y^2}{2t}}}{\sqrt{t}\int^{y/\sqrt{t}}_0e^{-u^2/2}du}\leq \frac{1}{y}, \ y>0, t>0.
$$
We get following inequalities
$$
\mathbb{E}\int^{T-\delta}_\delta \left( \partial_y\log\gamma_{\mathbb{R}_+}(T-s,Y_{y_n}(s))\right)^{3/2}ds\leq \int^{T-\delta}_\delta \mathbb{E}(Y_{y_n}(s))^{-3/2}ds=
$$
$$
= \int^{T-\delta}_\delta \frac{\mathbb{E}(B_{y_n}(s))^{-3/2}1_{\min_{r\in [0,s]}B_{y_n}(r)>0}\gamma_{\mathbb{R}_+}(T-s,B_{y_n}(s))}{\gamma_{\mathbb{R}_+}(T,y_n)}ds=
$$
$$
=\int^{T-\delta}_\delta\frac{\int^\infty_0 y^{-3/2}\gamma_{\mathbb{R}_+}(T-s,y)\frac{1}{\sqrt{2\pi s}}e^{-\frac{(y-y_n)^2}{2s}}(1-e^{-\frac{2yy_n}{s}})dy}{\int^\infty_0 \frac{1}{\sqrt{2\pi T}}e^{-\frac{(y-y_n)^2}{2T}}(1-e^{-\frac{2yy_n}{T}})dy}ds\leq 
$$
$$
\leq (T-2\delta) \sqrt{\frac{T}{\delta}} \frac{\int^\infty_0 y^{-3/2}e^{-\frac{(y-y_n)^2}{2(T-\delta)}}(1-e^{-\frac{2yy_n}{\delta}})dy}{\int^\infty_0 e^{-\frac{(y-y_n)^2}{2 T}}(1-e^{-\frac{2yy_n}{T}})dy}\xrightarrow[n\to\infty]{ } (T-2\delta)\left(\frac{T}{\delta}\right)^{3/2}\frac{\int^\infty_0 y^{-1/2} e^{-\frac{y^2}{2(T-\delta)}}dy}{\int^\infty_0 y e^{-\frac{y^2}{2T}dy}}.
$$
This proves \eqref{eq10_8_3}. Passing to the limit in \eqref{eq10_8_1} we get the relation 
$$
Y_0(t)=\int^t_0 \partial_y\log\gamma_{\mathbb{R}_+}(T-s,Y_{0}(s))ds +B_0(t).
$$
The weak existence for the problem \eqref{eq2} is proved. We prove the existence and uniqueness of the strong solution using the Yamada-Watanabe theorem \cite[Th. (1.7), Ch. IX]{RY}. Let $Y$ and $\tilde{Y}$ solve \eqref{eq2}. Then for almost all $t\in (0,T)$
$$
\frac{1}{2}\partial_t (Y(t)-\tilde{Y}(t))^2=\left(Y(t)-\tilde{Y}(t)\right)\left( \partial_y \log\gamma_{\mathbb{R}_+}(T-t,Y(t))- \partial_y \log\gamma_{\mathbb{R}_+}(T-t,\tilde{Y}(t))\right)\leq 0,
$$
since the function $y\to \gamma_{\mathbb{R}_+}(T-t,y)$ is log-concave. It follows that $Y(t)=\tilde{Y}(t)$ for all $t\in [0,T].$ The pathwise uniqueness of the problem \eqref{eq2} is proved.

\end{proof}

Next we derive two corollaries of the theorem. The first one is a straightforward generalization to the multidimensional case.

\begin{corollary}
\label{cor2} Let $x\in \mathbb{R}^d$ be arbitrary,   $e\in \mathbb{R}^d$ be a unit vector, and $H=\{y\in\mathbb{R}^d:(y-x)\cdot e>0\}.$ The statement of the theorem \ref{thm1} holds for $G=H$ and $x.$

\end{corollary}

In the next corollary we give a new proof of the well-known theorem on the equivalence between the distribution $\nu_{0,T}(\cdot,\mathbb{R}_+)$ of the Brownian  meander and the distribution $Q$ of the three-dimensional Bessel process. We recall that the three-dimensional Bessel process is defined as the process  $t\to \sqrt{B^2_1(t)+B^2_2(t)+B^2_3(t)},$ where $B_1,B_2,B_3$ are independent $\mathbb{R}-$valued Brownian motions started at zero. Consider the problem
\begin{equation}
\label{eq10_8_4}
\begin{cases}
dZ(t)=\frac{1}{Z(t)}dt +dW(t), \\
Z(0)=0, \\
Z(t)>0, t>0 
\end{cases}
\end{equation}
where $W$ is a standard $\mathbb{R}-$valued Brownian motion. This problem has a unique strong solution \cite{Cherny}, and its distribution coincides with $Q$. By $Q_T$ we  denote the distribution of the process $\{Z(t)\}_{t\in [0,T]}$ in $\mathcal{C}^1_T.$

\begin{corollary}
\label{cor3}\cite{Imhof} The measure $\nu_{0,T}(\cdot,\mathbb{R}_+)$ is equivalent to the distribution $Q_T$ of the three-dimensional Bessel process started at $0.$ The Radon-Nikodym density is given by 
$$
\frac{d \nu_{0,T}(\cdot,\mathbb{R}_+)}{dQ_t}(Z)= \frac{\sqrt{\pi T}}{\sqrt{2} Z(T) }.
$$

\end{corollary}

\begin{proof} The idea of the proof is to change the underlying probability measure $Q_T$ in order to convert the problem \eqref{eq10_8_4} to the problem \eqref{eq2}. A natural candidate for the density is given by the Girsanov theorem:
$$
\rho=\exp\bigg(\int^T_0 \left(\partial_y \log\gamma_{\mathbb{R}_+}(T-s,Z(s))-\frac{1}{Z(s)}\right)dW(s)-
$$
$$
-\frac{1}{2}\int^T_0 \left(\partial_y \log\gamma_{\mathbb{R}_+}(T-s,Z(s))-\frac{1}{Z(s)}\right)^2 ds\bigg).
$$
Because of singularities as $s\to 0$ and $s\to T$ it is not obvious that $\rho$ is well-defined and is a density. From \eqref{eq12_8_2} we have
$$
\partial_y \log\gamma_{\mathbb{R}_+}(s,y)=\frac{e^{-\frac{y^2}{2s}}}{\sqrt{s}\int^{\frac{y}{\sqrt{s}}}_0 e^{-\frac{u^2}{2}}du}.
$$
Elementary inequalities
$$
0\leq 1-\frac{ye^{-\frac{y^2}{2}}}{\int^y_0 e^{-\frac{u^2}{2}}du}\leq \frac{y^2}{2}.
$$
imply that the process
$$
X(t)=\partial_y \log\gamma_{\mathbb{R}_+}(T-t,Z(t))-\frac{1}{Z(t)}, 0\leq t<T,
$$
satisfies 
$$
|X(t)|= \frac{1}{Z(t)}\left(1-\frac{\frac{Z(t)}{\sqrt{T-t}}e^{-\frac{Z(t)^2}{2(T-t)}}}{\int^{\frac{Z(t)}{\sqrt{T-t}}}_0 e^{-\frac{u^2}{2}}du}\right)\leq 
$$
$$
\leq \frac{1}{Z(t)}\min\left(1,\frac{Z(t)^2}{2(T-t)}\right)\leq \max\left(1,\frac{1}{2(T-t)}\right), 0\leq t<T.
$$
In particular, there is no singularity as $s\to 0$ in the definition of $\rho.$ To deal with the singularity as $s\to T$ we consider the process
$$
\rho_t=\exp\bigg(\int^t_0 \left(\partial_y \log\gamma_{\mathbb{R}_+}(T-s,Z(s))-\frac{1}{Z(s)}\right)dW(s)- 
$$
$$
-\frac{1}{2}\int^t_0 \left(\partial_y \log\gamma_{\mathbb{R}_+}(T-s,Z(s))-\frac{1}{Z(s)}\right)^2 ds\bigg).
$$
Since Novikov's condition \cite[Prop. (1.15), Ch. VIII]{RY} holds for the process $X,$ the process $(\rho_t)_{0\leq t<T}$ is a martingale. Let us show that $(\rho_t)_{0\leq t<T}$ is  a uniformly integrable martingale, with 
$$
\lim_{t\to T}\rho_t=\frac{\sqrt{\pi T}}{\sqrt{2} Z(T) }.
$$
To this end consider the function
$$
b(t,y)=\log\left( \int^{\frac{y}{\sqrt{T-t}}}_0e^{-\frac{u^2}{2}}du\right)-\log (y).
$$ 
It has the following limit values:
\begin{equation}
\label{27_8_eq2}
\lim_{t\to 0,y\to 0}b(t,y)=-\log\sqrt{T}, \lim_{t\to T, y\to z}b(t,y)=\log\sqrt{\frac{\pi}{2}}-\log z,
\end{equation}
where $z>0$ is arbitrary. Further, we have 
$$
\partial_t b(t,y)=\frac{y e^{-\frac{y^2}{2(T-t)}}}{2(T-t)^{\frac{3}{2}}\int^{\frac{y}{\sqrt{T-t}}}e^{-\frac{u^2}{2}}du},
$$
$$
\partial_y b(t,y)=\frac{e^{-\frac{y^2}{2(T-t)}}}{\sqrt{T-t}\int^{\frac{y}{\sqrt{T-t}}}_0e^{-\frac{u^2}{2}}du}-\frac{1}{y},
$$
$$
\partial^2_y b(t,y)=-\frac{ye^{-\frac{y^2}{2(T-t)}}}{(T-t)^{\frac{3}{2}}\int^{\frac{y}{\sqrt{T-t}}}_0e^{-\frac{u^2}{2}}du}-\frac{e^{-\frac{y^2}{T-t}}}{(T-t)\left(\int^{\frac{y}{\sqrt{T-t}}}_0e^{-\frac{u^2}{2}}du\right)^2}+\frac{1}{y^2}.
$$
By the It\^o formula,
$$
d b\left(t,Z(t)\right)=\frac{Z(t) e^{-\frac{Z(t)^2}{2(T-t)}}}{2(T-t)^{\frac{3}{2}}\int^{\frac{Z(t)}{\sqrt{T-t}}}e^{-\frac{u^2}{2}}du} dt+
$$
$$
+ \left(\frac{e^{-\frac{Z(t)^2}{2(T-t)}}}{\sqrt{T-t}\int^{\frac{Z(t)}{\sqrt{T-t}}}_0e^{-\frac{u^2}{2}}du}-\frac{1}{Z(t)}\right)\left(\frac{1}{Z(t)}dt+dW(t)\right)+
$$
$$
+\frac{1}{2}\left(-\frac{Z(t)e^{-\frac{Z(t)^2}{2(T-t)}}}{(T-t)^{\frac{3}{2}}\int^{\frac{Z(t)}{\sqrt{T-t}}}_0e^{-\frac{u^2}{2}}du}-\frac{e^{-\frac{Z(t)^2}{T-t}}}{(T-t)\left(\int^{\frac{Z(t)}{\sqrt{T-t}}}_0e^{-\frac{u^2}{2}}du\right)^2}+\frac{1}{Z(t)^2}\right)dt=
$$
$$
=\left(\frac{e^{-\frac{Z(t)^2}{2(T-t)}}}{\sqrt{T-t}\int^{\frac{Z(t)}{\sqrt{T-t}}}_0e^{-\frac{u^2}{2}}du}-\frac{1}{Z(t)}\right)dW(t)-\frac{1}{2}\left(\frac{e^{-\frac{Z(t)^2}{2(T-t)}}}{\sqrt{T-t}\int^{\frac{Z(t)}{\sqrt{T-t}}}_0e^{-\frac{u^2}{2}}du}-\frac{1}{Z(t)}\right)^2dt.
$$
By \eqref{27_8_eq2},
$$
\lim_{t\to 0}b\left(t,Z(t)\right)=-\log\sqrt{T}, \ \lim_{t\to T}b\left(t,Z(t)\right)=\log\sqrt{\frac{\pi}{2}}-\log Z(T).
$$
Hence,
$$
\rho_t=\exp\left(b(t,Z(t))+\log\sqrt{T}\right)\to \frac{\sqrt{\pi T}}{\sqrt{2} Z(T) }, \ t\to T.
$$
By the Girsanov theorem, under the measure $\rho dQ_T$ the process
$$
\tilde{W}(t)=W(t)-\int^t_0 \left(\partial_y \log\gamma_{\mathbb{R}_+}(T-s,Z(s))-\frac{1}{Z(s)}\right)ds, \ 0\leq t<T
$$
is a Brownian motion. Hence, under the measure $\rho dQ_T,$ the process $\{Z(t)\}_{0\leq t\leq T}$ is a solution of the SDE
$$
dZ(t)=\frac{1}{Z(t)}dt+d\tilde{W}(t)+\left(\partial_y \log\gamma_{\mathbb{R}_+}(T-t,Z(t))-\frac{1}{Z(t)}\right)dt=
$$
$$
=\partial_y \log\gamma_{\mathbb{R}_+}(T-t,Z(t)) dt +d\tilde{W}(t),
$$
and thus is a Brownian meander.

\end{proof}

\section{Proof of the Theorem \ref{thm1}}

\begin{proof}

Given an open set $A\subset \mathbb{R}^d$ and a continuous function $f\in \mathcal{C}^d_T$ we will denote by $\tau_A(f)$ the first exit time 
$$
\tau_A(f)=\inf\{t>0: f(t)\not\in A\}.
$$
We recall that the set $G$ is assumed to be convex with a $C^2$ boundary in the neighborhood of its boundary point $x.$ Let us choose a unit vector $e\in \mathbb{R}^d$ and $r>0$ such that  $B(x+re,r)\subset G.$ Consider the half-space 
$$
H=\{y\in\mathbb{R}^d:(y-x)\cdot e>0\},
$$
so that 
$$
B(x+re,r)\subset G\subset H.
$$
Consider an auxiliary measure $\nu_{x,T}(\cdot;H)$ (see corollary \ref{cor3}). The corresponding process can be described as follows. Choose an orthonormal basis $\{e_1,\ldots,e_d\}$ in $\mathbb{R}^d,$ such that $e_1=e.$ Let $\{\tilde{Y}_1(t)\}_{0\leq t\leq T}$ be a Brownian meander, and $\{(\tilde{W}_2(t),\ldots,\tilde{W}_d(t))\}_{0\leq t\leq T}$ be a $\mathbb{R}^{d-1}-$valued Brownian motion independent from $\tilde{Y}_1$. Then $\nu_{x,T}(\cdot;H)$ is  the distribution of the process 
$\{x + \tilde{Y}_1(t)e_1+\sum^d_{i=2}\tilde{W}_i(t)e_i\}_{0\leq t\leq T}.$ 

By the corollary \ref{cor2}   $\nu_{x,T}(\cdot;H)$ is the distribution of the solution of the problem 
\begin{equation}
\label{eq17_8_1}
\begin{cases}
dY(t) =\nabla_y\log\gamma_H(T-t,Y(t))dt+dW(t) \\ 
Y(0)=x, \\ 
Y(t)\in H \ \mbox{ for a.a. } t\in(0,T)
\end{cases}  
\end{equation}
where $W$ is an $\mathbb{R}^d-$valued Brownian motion. By corollary \ref{cor3} the measure $\nu_{x,T}(\cdot;H)$ is equivalent to the distribution of the process $\{x + \tilde{Z}_1(t)e_1+\sum^d_{i=2}\tilde{W}_i(t)e_i\}_{0\leq t\leq T},$ where $\{\tilde{Z}_1(t)\}_{t\geq 0}$ is a three-dimensional Bessel process  independent from $\{(\tilde{W}_2(t),\ldots,\tilde{W}_d(t))\}_{0\leq t\leq T}.$  Applying \cite[Th. 3.4]{Burdzy} we deduce 
$$
\nu_{x,T}(\tau_{B(x+re,r)}(Y)>0;H)=1.
$$
Consequently,
$$
\nu_{x,T}(\{\tau_G(Y)>T\};H)>0
$$
and we can represent the measure $\nu_{x,T}(\cdot;G)$ via the density with respect to the measure $\nu_{x,T}(\cdot;H)$ (see \cite{Garbit} for the details):
$$
\frac{d\nu_{x,T}(\cdot;G)}{d\nu_{x,T}(\cdot;H)}=\frac{1_{\tau_G(Y)>T}}{\nu_{x,T}(\{\tau_G(Y)>T\};H)}.
$$
Let us apply the Girsanov theorem to this density. Introduce the function 
$$
\theta(t,y)=\mathbb{P}(\forall r\in[t,T] \ Y(r)\in G|Y(t)=y)=\frac{\gamma_G(T-t,y)}{\gamma_H(T-t,y)}, \ y\in G, 0\leq t<T.
$$
As in the proof of theorem \ref{thm3}, an application of the It\^o formula implies the Clark representation
$$
1_{\tau_G(Y)>T}=\theta(0,x)+\int^T_0 1_{\tau_G(Y)>s}(\nabla_y \theta(s,Y(s)),dW(s)).
$$
By the Markov property, we have 
$$
\mathbb{E}[1_{\tau_G(Y)>T}|\mathcal{F}_s]=1_{\tau_G(Y)>s}\theta(s,Y(s)).
$$
Repeating arguments of the theorem \ref{thm3}, under the measure $\nu_{x,T}(\cdot;G)$ the process 
$$
\tilde{W}(t)=W(t)-\int^t_0 \nabla_y \log \theta(s,Y(s))ds, 0\leq t\leq T,
$$
is a Brownian motion. From \eqref{eq17_8_1} we deduce that under the measure $\nu_{x}(\cdot;G)$ the process $Y$ satisfies the equation
$$
dY(t)=\nabla_y\log\gamma_H(T-t,Y(t))dt+\nabla_y \log \theta(t,Y(t))dt+d\tilde{W}(t)=
$$
$$
=\nabla_y \log \gamma_G(T-t,Y(t))dt+\tilde{W}(t).
$$

It remains to check pathwise uniqueness for the problem \eqref{eq1}. Let $Y$ and $\tilde{Y}$ solve \eqref{eq1}. Then 
$$
\frac{1}{2}\partial_t |Y(t)-\tilde{Y}(t)|^2=\left(Y(t)-\tilde{Y}(t),\nabla_y\log\gamma_G(T-t,Y(t))-\nabla_y\log\gamma_G(T-t,\tilde{Y}(t))\right)\leq 0,
$$
where the last inequality follows from log-concavity of the function $y\to \gamma_G(T-t,y)$ \cite{BL}.

\end{proof}

\section{Clusters in coalescing stochastic flows}

\subsection{Arratia flow}

In this section we will use duality theory for coalescing stochastic flows on the real line developed in \cite{Riabov_duality}. By a backward stochastic flow we will understand a family $\{\phi_{t,s}:-\infty<s\leq t<\infty\}$ of measurable random mappings of $\mathbb{R},$ such that the family $\{\hat{\phi}_{s,t}=\phi_{-s,-t}:-\infty<s\leq t<\infty\}$ is a stochastic flow. Let $\psi=\{\psi_{s,t}:-\infty<s\leq t<\infty\}$ be the Arratia flow. A dual flow $\tilde{\psi}=\{\tilde{\psi}_{t,s}:-\infty<s\leq t<\infty\}$ is defined as a backward stochastic flow whose trajectories do not cross trajectories of the flow $\psi,$ i.e. for all $s\leq t, x,y\in\mathbb{R}$ and $\omega\in\Omega$
$$
(\psi_{s,t}(\omega,x)-y)(x-\tilde{\psi}_{t,s}(\omega,y))\geq 0.
$$
For the needed properties of the Arratia flow as well as for existence and properties of its dual we refer to \cite{Riabov_RDS, Riabov_duality}. In particular, we recall that the dual $\tilde{\psi}$ of the Arratia flow $\psi$ is itself the Arratia flow (with time reversed). As it was mentioned in the Introduction, the image $\psi_{0,T}(\mathbb{R})$ is a locally finite subset of $\mathbb{R}$ unbounded from below and from above. Let us fix $\omega$ for a while. With every point $\zeta\in \psi_{0,T}(\omega,\mathbb{R})$ we associate a cluster 
$$
K_\zeta=\cup_{t\in [0,T]}\{(T-t,x):\psi_{T-t,T}(\omega, x)=\zeta\}.
$$
By $\alpha_\zeta$ and $\beta_\zeta$ we denote the lower and the upper boundaries of the cluster $K_\zeta:$
$$
\alpha_\zeta(t)=\inf\{x\in \mathbb{R}:(T-t,x)\in K_\zeta\}, \ \beta_\zeta(t)=\sup\{x\in \mathbb{R}:(T-t,x)\in K_\zeta\}.
$$
This natural definition of $\alpha_\zeta, \beta_\zeta$ is not a rigorous definition of a stochastic processes, as the choice of the random quantity $\zeta$ is not specified. In the following lemma we overcome this issue and simultaneously define the conditional distribution of boundary processes conditioned on the event $\{\zeta=x\}.$ 

\begin{lemma}
\label{lem17_1} With probability 1 for all $\zeta\in \psi_{0,T}(\mathbb{R})$
$$
\lim_{x\to \zeta+}\sup_{t\in [0,T]}|\tilde{\psi}_{T,T-t}(x)-\beta_\zeta(t)|=0,
$$
and
$$
\lim_{x\to \zeta-}\sup_{t\in [0,T]}|\tilde{\psi}_{T,T-t}(x)-\alpha_\zeta(t)|=0.
$$
\end{lemma}

\begin{proof}
For continuity of $\alpha_\zeta,$ $\beta_\zeta$ we refer to \cite{Riabov_duality}.  Let $x>\zeta$ and $t\in [0,T].$ If $\tilde{\psi}_{T,T-t}(x)<\beta_\zeta(t),$ then there exists $y>\tilde{\psi}_{T,T-t}(x)$ such that 
$$
\psi_{T-t,T}(y)=\zeta<x,
$$
which contradicts duality. So, for all $x>\zeta$ and all $t\in [0,T],$
$$
\beta_\zeta(t)\leq \tilde{\psi}_{T,T-t}(x).
$$
It remains to check that for all $t\in[0,T]$
$$
\inf_{x>\zeta} \tilde{\psi}_{T,T-t}(x)=\beta_\zeta(t).
$$
Assume that $\inf_{x>\zeta} \tilde{\psi}_{T,T-t}(x)>\beta_\zeta(t)$ and let $y\in (\beta_\zeta(t),\inf_{x>\zeta} \tilde{\psi}_{T,T-t}(x)).$ For every $x>\zeta$ duality implies that 
$\psi_{T-t,T}(y)\leq x.$ Hence, $\psi_{T-t,T}(y)\leq \zeta.$ But the latter contradicts $y>\beta_\zeta(t).$ The proof for $\alpha_\zeta$ is similar.

\end{proof}

Observe the equality of events 
$$
\{(u,v)\cap \psi_{0,T}(\mathbb{R})\ne \emptyset\}=\{\tilde{\psi}_{T,0}(u)<\tilde{\psi}_{T,0}(v)\},
$$
the latter event being the event that two independent $\mathbb{R}-$valued Brownian motions started at $u$ and $v$ and haven't met during the time $T$. Combining this consideration with results of Lemma \ref{lem17_1} and Theorem \ref{thm1}, we get the corollary.

\begin{corollary}
\label{cor17_1} 
Conditional distribution of the process 
$$
\{(\tilde{\psi}_{T,T-t}(u),\tilde{\psi}_{T,T-t}(v))\}_{t\in [0,T]}
$$
conditionally on the event $\{(u,v)\cap \psi_{0,T}(\mathbb{R})\ne \emptyset\}$ weakly converge as $u\to x-, v\to x+$ to the solution of the problem 
$$
\begin{cases}
dY(t)=\nabla_y\log\gamma_H(T-t,Y(t))dt+dW(t), \\
Y(0)=(x,x), \\
Y(t)\in H \mbox{ for a.a. } t\in (0,T),
\end{cases}
$$
where $H=\{y\in\mathbb{R}^2:y_1<y_2\},$ $W$ is a standard $\mathbb{R}^2-$valued Brownian motion, and $\gamma_H$ is defined in \eqref{28_08_eq3}.

\end{corollary}

Direct computation gives $\gamma_H(t,y)=\sqrt{\frac{2}{\pi}}E\left(\frac{y_2-y_1}{\sqrt{2t}}\right),$ where $E(x)=\int^x_0 e^{-\frac{u^2}{2}}du.$ Consequently, we can identify the conditional law of boundaries $(\alpha_\zeta,\beta_\zeta)$ given that $\{\zeta=x\}$ via the problem
$$
\begin{cases}
d\alpha_\zeta(t)=-\frac{e^{-\frac{(\beta_\zeta(t)-\alpha_\zeta(t))^2}{4(T-t)}}}{\sqrt{2(T-t)}E(\frac{\beta_\zeta(t)-\alpha_\zeta(t)}{\sqrt{2(T-t)}})}dt+dW_1(t) \\
d\beta_\zeta(t)=\frac{e^{-\frac{(\beta_\zeta(t)-\alpha_\zeta(t))^2}{4(T-t)}}}{\sqrt{2(T-t)}E(\frac{\beta_\zeta(t)-\alpha_\zeta(t)}{\sqrt{2(T-t)}})}dt+dW_2(t) \\
\alpha_\zeta(0)=\beta_\zeta(0)=x \\
\alpha_\zeta(t)<\beta_\zeta(t) \mbox{ for a.a. } t\in (0,T)
\end{cases}.
$$

\subsection{Arratia flow with drift}

In this section the developed approach is adapted to the unbounded cluster in the Arratia flow with drift. Let $a:\mathbb{R}\to\mathbb{R}$ be a Lipschitz function. Consider a SDE
\begin{equation}
\label{4_8_eq1}
dX(t)=a(X(t))dt+dw(t),
\end{equation}
where $w$ is a Wiener process. Informally, the Arratia flow with drift describes the joint motion of solutions of the equation \eqref{4_8_eq1} that start from all points of the real line at every moment of time, move independently before the meeting time and coalesce at the meeting time. Precisely, we say that a coalescing stochastic flow $\psi=\{\psi_{s,t}:-\infty<s\leq t<\infty\}$ is the Arratia flow with drift $a,$ if the following condition is satisfied:

For any $n\geq 1$ and  $x_1<\ldots<x_n$  let $\{(X_1(t),\ldots,X_n(t))\}_{t\geq 0}$ be the solution of the problem 
$$
\begin{cases}
dX_i(t)=a(X_i(t))dt+dW_i(t) \\
X_i(0)=x_i
\end{cases}, \ 1\leq i\leq n,
$$
where $W_1,\ldots,W_n$ are independent standard $\mathbb{R}-$valued Brownian motions. Denote  $\sigma=\inf\{t\geq 0: \exists i\ne j \ X_i(t)=X_j(t)\}.$ Further, let $\{(\psi_{s,s+t}(x_1),\ldots,\psi_{s,s+t}(x_n))\}_{t\geq 0}$ be the $n-$point motion of the flow started at time $s$ from points $x_1,\ldots,x_n.$ Denote $\tau=\inf\{t\geq 0: \exists i\ne j \ \psi_{s,s+t}(x_i)=\psi_{s,s+t}(x_j)\}.$ Then $\mathbb{R}^n-$valued processes
$$
t\to (\psi_{s,s+t\wedge \tau}(x_1),\ldots,\psi_{s,s+t\wedge \tau}(x_n))
$$
and 
$$
t\to (X_1(t\wedge \sigma),\ldots, X_n(t\wedge \sigma))
$$
are identically distributed. 

For the existence of the Arratia flow with drift we refer to \cite{Riabov_RDS}. When the drift $a$ is strictly monotone, an infinite cluster arises in the flow $\psi.$

\begin{theorem}
\label{thm18_1} \cite{DRS}
Let $\psi$ be the Arratia flow with drift $a.$ Assume that the drift $a$ is Lipschitz and for some $\lambda>0$ and all $x,y\in \mathbb{R}$ one has
$$
(a(x)-a(y))(x-y)\leq -\lambda (x-y)^2.
$$
Then there exists a unique stationary process $(\eta_t)_{t\in \mathbb{R}}$ such that for all $s\leq t$ and all $\omega$
$$
\psi_{s,t}(\omega,\eta_s(\omega))=\eta_t(\omega).
$$
\end{theorem}

Further we assume that the drift $a$ satisfies assumptions of the theorem \ref{thm18_1}. The process $(\eta_t)_{t\geq 0}$ represents the motion of a stationary point in the flow. In particular, the one-dimensional distribution of $(\eta_t)_{t\geq 0}$ is given by the stationary distribution of the equation \eqref{4_8_eq1}: 
$$
\mathbb{P}(\eta_t\in \Delta)=C\int_{\Delta}e^{2\int^x_0 a(y)dy}dx,
$$
where $C=\left(\int^\infty_{-\infty}e^{2\int^x_0 a(y)dy}dx \right)^{-1}.$ An infinite cluster can be associated with $\eta_0.$ Namely, at every moment $t\geq 0$ there exists an interval of points that have coalesced into $\eta_0$ at time $0:$
$$
K_0(t)=\{x\in \mathbb{R}:\psi_{-t,0}(x)=\eta_0\}, \ t\geq 0.
$$
The set $K_0=\cup_{t\geq 0}(\{-t\}\times K_0(t))$ will be called the cluster with the vertex $\eta_0.$ Let us introduce boundary processes
$$
\alpha_0(t)=\inf K_0 (t), \ \beta_0(t)=\sup K_0 (t).
$$
The theorem \ref{thm18_2} describes the conditional distribution of processes $(\alpha_0(t),\beta_0 (t))$ conditioned on the event $\{\eta_0=x\}.$ The following analogue of the lemma \ref{lem17_1} follows from properties of the dual flow $\{\tilde{\psi}_{t,s}:-\infty<s\leq t<\infty\}$ obtained in \cite{Riabov_duality}.

\begin{lemma}
\label{lem18_1}   With probability 1 for all $T\geq 0$
$$
\lim_{x\to \eta_0+}\sup_{t\in [0,T]}|\tilde{\psi}_{0,-t}(x)-\beta_0(t)|=0,
$$
and
$$
\lim_{x\to \eta_0-}\sup_{t\in [0,T]}|\tilde{\psi}_{0,-t}(x)-\alpha_0(t)|=0.
$$
\end{lemma}

In \cite{DRS} it was proved that with probability 1, 
$$
\lim_{t\to \infty}\beta_0(t)=\infty, \lim_{t\to \infty}\alpha_0(t)=-\infty.
$$
Hence, the following equality of events holds:
$$
\{u<\eta_0<v\}=\{\lim_{t\to \infty}\tilde{\psi}_{0,-t}(u)=-\infty, \lim_{t\to \infty}\tilde{\psi}_{0,-t}(v)=\infty \}
$$
Let 
$$
\theta(y_1,y_2)=\mathbb{P}(\eta_0\in (y_1,y_2))=\int^{y_2}_{y_1} \pi(x)dx.
$$

\begin{theorem}
\label{thm18_2} 
Conditional distribution of the process 
$$
\{(\tilde{\psi}_{0,-t}(u),\tilde{\psi}_{0,-t}(v))\}_{t\geq 0}
$$
conditionally on the event $\{u<\eta_0<v\},$  weakly converge as $u\to x-, v\to x+$ to the solution of the problem 
\begin{equation}
\label{eq18_1}
\begin{cases}
dY_1(t)=\left(-a(Y_1(t))+\frac{\partial \log\theta(Y_1(t),Y_2(t))}{\partial y_1}\right)dt+dW_1(t), \\
dY_2(t)=\left(-a(Y_2(t))+\frac{\partial \log\theta(Y_1(t),Y_2(t))}{\partial y_2}\right)dt+dW_2(t), \\
Y_1(0)=Y_2(0)=x, \\
Y_1(t)<Y_2(t) \mbox{ for a.a. } t>0,
\end{cases}
\end{equation}
where $W$ is a standard $\mathbb{R}^2-$valued Brownian motion.

\end{theorem}

\begin{proof} The dual process $\tilde{\psi}$ is the Arratia flow with drift $-a,$ see \cite{Riabov_duality, DRS}. Let $\{(Y_1(t),Y_2(t))\}_{t\geq 0}$ be a solution of the SDE 
$$
\begin{cases}
dY_1(t)=-a(Y_1(t))dt+dW_1(t), \\
dY_2(t)=-a(Y_2(t))dt+dW_2(t), \\
Y_1(0)=u, Y_2(0)=v
\end{cases}
$$
where $W$ is a standard $\mathbb{R}^2-$valued Brownian motion. The law of the process 
$$
\{(\tilde{\psi}_{0,-t}(u),\tilde{\psi}_{0,-t}(v))\}_{t\geq 0}
$$
conditioned on the event $\{u<\eta_0<v\}$ coincides with the law of the process $Y$ conditioned on the event 
$$
A=\{\forall t\geq 0 \ Y_1(t)<Y_2(t), \lim_{t\to \infty}Y_1(t)=-\infty, \lim_{t\to \infty}Y_2(t)=\infty\}.
$$
Let $\sigma=\inf\{t\geq 0: Y_1(t)=Y_2(t)\}.$ Applying arguments from the proof of Theorem \ref{thm3} to the process 
$$
t\to \theta(Y_1(t\wedge \sigma),Y_2(t\wedge \sigma)),
$$
we get the Clark representation
$$
1_A=\theta(u,v)+\int^\tau_0 (\nabla \theta(Y(s)),dW(s)).
$$
By the Markov property,
$$
\mathbb{E}[1_A|Y(s),s\leq t]=1_{\tau>t}\theta(Y(t)).
$$
The Girsanov theorem implies that with respect to the law of $Y$ conditioned on the event $A,$ the process 
$$
\tilde{W}(t)=W(t)-\int^t_0 \nabla \log\theta(Y(s))ds, \  t\geq 0, 
$$
is a Brownian motion. This implies equations  \ref{eq18_1}  for the distribution of the process   $\{(\tilde{\psi}_{0,-t}(u),\tilde{\psi}_{0,-t}(v))\}_{t\geq 0}$  conditioned on the event $\{u<\eta_0<v\}.$

\end{proof}

\end{document}